\documentclass[12pt]{amsart}
\usepackage[T1]{fontenc}
\usepackage{lmodern}
\usepackage{microtype}
\usepackage{amsmath,amscd,amssymb,mathtools,mathrsfs}
\usepackage[graph,frame,poly,arc]{xy}
\usepackage{pgf,tikz}
\usetikzlibrary{automata,arrows}
\usepackage[plainpages,urlcolor=blue]{hyperref}

\theoremstyle{plain}
\newtheorem{thm}[subsection]{Theorem}

\newtheorem{cor}[subsection]{Corollary}

\newtheorem{lemma}[subsection]{Lemma}

\theoremstyle{definition}
\newtheorem{rk}[subsection]{Remark}
\newtheorem{definition}[subsection]{Definition}
\newtheorem{ex}[subsection]{Example}

\theoremstyle{remark}
\newtheorem{remark}[subsection]{Remark}

\numberwithin{equation}{section}
\newcommand{\OO}{{\mathcal O}}

\newcommand{\A}{{\mathcal A}}

\newcommand{\B}{{\mathcal B}}

\newcommand{\R}{\mathbb{R}}
\newcommand{\C}{\mathbb{C}}

\newcommand{\PP}{\mathbb{P}}

\DeclareMathOperator{\rank}{rank}

\DeclareMathOperator{\indeg}{indeg}

\begin{document}

\title[On Ziegler pairs of arrangements]
{On Ziegler pairs of line arrangements: from non-existence to abundance}

\author[Alexandru Dimca and Piotr Pokora
]{Alexandru Dimca and Piotr Pokora
}



\begin{abstract}
We study Ziegler pairs of line arrangements from both numerical and homological perspectives. We distinguish classical Ziegler pairs, for which the minimal degree of a Jacobian relation denoted by ${\rm mdr}$ differs, from the more general notion detected by distinct graded Betti data. First, we show that for arrangements of $d<9$ lines the intersection lattice determines ${\rm mdr}$, and hence classical Ziegler pairs do not occur in this range. Then we list several distinct Ziegler pairs with $d=10$ in the broader sense. In particular, we construct higher-degree examples with the same intersection lattice, the same minimal degree of a Jacobian relation, and the same Hilbert function of the Milnor algebra, but with different minimal graded free resolutions.
Another rather surprising consequence is that being maximal Tjurina for a line arrangement is not determined by the combinatorics.
\end{abstract}

\maketitle

\section{Introduction}

The purpose of this paper is to discuss Ziegler pairs of line arrangements from two complementary points of view. The first one is classical and concerns whether the intersection lattice determines the minimal degree of a Jacobian relation. The second one is homological and concerns pairs with the same intersection lattice and the same minimal degree, but with different minimal graded free resolutions of the corresponding Milnor algebras.

Let $\A:f=0$ be a line arrangement in $\PP^2$, where $d=\deg f$ is the number of lines, and let $m=m(\A)$ be the maximal multiplicity of an intersection point of $\A$. We write
\[
        r=r(\A)={\rm mdr}(f)=\indeg D_0(f)
\]
for the minimal degree of a Jacobian syzygy. If $\A$ is essential, equivalently if not all lines of $\A$ are concurrent, and
\[
        e_1\leq e_2\leq \cdots \leq e_s
\]
denotes the ordered sequence of exponents of $\A$, then $r=e_1$ and the upper bound
\[
        e_s\leq d-2
\]
follows from \cite{Sch}.

In this paper a Ziegler pair, in the general sense, is a pair of line arrangements $(\A,\A')$ with isomorphic intersection lattices but different graded Betti data for their Jacobian algebras. We call such a pair \textbf{classical} if
\[
        {\rm mdr}(\A)\neq {\rm mdr}(\A').
\]
Thus the original examples of Ziegler and Yuzvinsky are classical Ziegler pairs, whereas the homological pairs constructed later in this paper may have the same value of ${\rm mdr}$.

The first part of the paper shows that classical Ziegler pairs cannot occur for $d<9$. More precisely, for arrangements with at most eight lines the intersection lattice determines ${\rm mdr}$. We do not claim that the full exponent sequence is determined in this range: in degree $8$ some fixed numerical types admit more than one possible degree sequence while keeping the same first exponent. After the elementary cases are removed, the only cases requiring separate analysis are
\[
        (d,m)=(6,3),\quad (7,3),\quad (8,4),\quad (8,3),
\]
and these are treated in Theorems~\ref{thmA}, \ref{thmB}, \ref{thmC}, and \ref{thmD}.

In degree $d=9$, two classical Ziegler pairs are currently known: the examples due to Ziegler and Yuzvinsky \cite{Zi,Yu}, and a new pair recently constructed by the authors \cite{nine}. Since the analysis of this borderline case is very complicated and requires a separate treatment, we postpone its detailed discussion to a forthcoming research project. On the other hand, for $d>9$ Ziegler pairs in the broader homological sense seem to abound.

The second part of the paper gives new examples in which the difference is subtler. We construct a two-parameter family of arrangements of eleven lines with constant intersection lattice. At the concrete parameter value
\[
q=(31,17),
\]
the Jacobian syzygies have degree sequence
\[
        (6,7,7,8),
\]
whereas at the special point
\[
        p=\left(\frac{3}{10},\frac{2}{7}\right)
\]
the intersection lattice is unchanged and the degree sequence becomes
\[
        (6,7,7,8,8).
\]
The two Milnor algebras have the same Hilbert function, but their minimal graded free resolutions are different. Deleting a suitable line from both arrangements then gives a Ziegler pair of ten lines with degree sequences
\[
        (6,6,6,7)
        \qquad\text{and}\qquad
        (6,6,6,7,7).
\]
Thus the pair has the same value of ${\rm mdr}$ and the same Hilbert function, and the distinction is purely homological.

The paper is organized as follows. Section~2 recalls the algebraic notation and the preliminary tools used throughout the paper. Sections~3 and~4 prove the non-existence of classical Ziegler pairs with at most eight lines while recording the possible exponent sequences. Section~5 introduces the eleven-line family and the rank-drop mechanism producing the homological jump. Section~6 applies deletion to obtain new ten-line Ziegler pairs. Section~7 recalls the known examples with nine and ten lines and places the new examples in this list. We also show by an example that the property of being maximal Tjurina, in some sense a property very similar to being free, is not determined by the combinatorics, see \eqref{eqMTC} for the definition of maximal Tjurina curves and  Remark \ref{rkMT} for details.

\section{Preliminaries}

Let $S=\C[x,y,z]$ be the standard graded polynomial ring, and let $C:f=0\subset \PP^2$ be a reduced plane curve of degree $d$.  In the case of a line arrangement we write
\[
        f=\ell_1\cdots \ell_d,
\]
where the $\ell_i\in S_1$ are pairwise non-proportional linear forms.

The Jacobian ideal of $f$ is
\[
        J_f=(f_x,f_y,f_z)\subset S,
\]
and the Milnor algebra of $f$ is
\[
        M(f)=S/J_f.
\]
The module of Jacobian syzygies is
\[
        AR(f)=\{(a,b,c)\in S^3\mid af_x+bf_y+cf_z=0\}.
\]
Equivalently, we use the notation $D_0(f)$ for the corresponding module of logarithmic derivations annihilating $f$; throughout the paper we identify $D_0(f)$ with $AR(f)$ in the usual way.  The minimal degree of a Jacobian relation is
\[
        {\rm mdr}(f)=\min\{q\geq 0\mid AR(f)_q\neq 0\}=\indeg D_0(f).
\]
If $AR(f)$ has homogeneous minimal generators of degrees $e_1\leq\cdots\leq e_s$, we call $(e_1,\ldots,e_s)$ the degree sequence, or the sequence of exponents, of the arrangement.

For a line arrangement $\A$, we denote by $n_k=n_k(\A)$ the number of points where exactly $k$ lines of $\A$ meet.  Since all singularities are ordinary quasihomogeneous multiple points, the total Tjurina number is
\[
        \tau(\A)=\sum_{k\geq 2} n_k(k-1)^2.
\]
The intersection lattice $L(\A)$ determines all numbers $n_k$, but in general it need not determine the module $D_0(f)$ or the minimal free resolution of $M(f)$.

\begin{definition}
\label{defZP}
Two line arrangements $\A:f=0$ and $\A':f'=0$ in $\PP^2$ are said to form a \emph{Ziegler pair} if they have isomorphic intersection lattices $L(\A)$ and $L(\A')$, obtained from the corresponding central arrangements in $\C^{3}$, and distinct homological data, i.e., distinct graded Betti tables of the minimal resolutions of their Jacobian algebras $M(f)$ and $M(f')$.

A Ziegler pair is called a \emph{classical Ziegler pair} if
\[
        {\rm mdr}(f)\neq {\rm mdr}(f').
\]

A Ziegler pair $\A:f=0$ and $\A':f'=0$ may have one or several of the following additional properties.
\begin{enumerate}
\item We say that the Ziegler pair satisfies condition ${\rm (HF)}$ if the Jacobian algebras $M(f)$ and $M(f')$ have the same Hilbert function, that is
$$\dim M(f)_k =\dim M(f')_k \text{ for any integer } k.$$
\item We say that the Ziegler pair satisfies condition ${\rm (MDR)}$ if
\[
        {\rm mdr}(f)={\rm mdr}(f').
\]
Thus a Ziegler pair satisfying ${\rm (MDR)}$ is, by definition, not classical.
\item We say that the Ziegler pair satisfies condition $({\rm SPEC}_0)$ if the arrangement $\A$ is a specialization of the arrangement $\A'$. More precisely, this means that there exists an interval $I=[0,t_0]\subset \R$ and a smooth family of homogeneous polynomials $f_t$, for $t\in I$, such that
\[f_0=f, \qquad f_{t_0}=f',\]
and the corresponding family of hyperplane arrangements $\A_t:f_t=0$ has constant intersection lattice, i.e.
\[L(\A_t) \cong L(\A_{t'}) \quad \text{for all } t,t'\in I.\]
If, in addition, the graded Betti numbers of the minimal resolutions of the Jacobian algebras $M(f_t)$ are constant for all $0<t\leq t_0$, then we say that the Ziegler pair $\A:f=0$ and $\A':f'=0$ satisfies condition $({\rm SPEC})$.
\end{enumerate}
\end{definition}
As an example, the classical Ziegler pair in degree $d=9$ is classical in the above sense and satisfies condition $({\rm SPEC})$.

We shall use the following standard bounds and addition-deletion tools.

Let us recall that by  \cite[Theorem 1.2]{Dca} one has 
\begin{equation} 
\label{eqR}
r=d-m \text{ if } 2m>d, \text{ or } m-1 \leq r \leq d-m \text{ otherwise}.
\end{equation}

We list first the cases where $r$ is clearly determined by the pair $(d,m)$ and some other invariants of the intersection lattice of $\A$.

\noindent {\bf Case 1:} $2m > d$. Then it follows from \eqref{eqR}
that $r=d-m$.

\bigskip

\noindent {\bf Case 2:} 
When $m=2$, it follows from \cite{DStEdin}
that $r=d-2$. In fact, in this case  for $d\geq 4$, $\A$ is a maximal Tjurina curve of type
$(d,d-2)$, see \cite[Proposition 5.11]{maxT}, and hence all the numerical invariants of the minimal resolution of $D_0(f)$ are determined by $d$.

\bigskip

In the sequel we assume $m>2$ and $d \geq 2m \geq 6$, since all the other cases are covered by Case 1 and Case 2 above.
Moreover, it follows from \cite[Theorem 2.1]{DS14} that
\begin{equation} 
\label{eqMS}
r \geq \frac{2d}{m}-2.
\end{equation} 

For any reduced plane curve $C:f=0$,
let $J=J_f$ be the Jacobian ideal of $f$ and $I=I_f$ be its saturation with respect to the maximal ideal $(x,y,z)$. We recall that
$$N(f)=I/J$$
is called the Jacobian module of the curve $C$. Then we have the following result, see \cite[Theorem 6.2]{DIS}.

\begin{thm}
\label{thm2}
Let $C_1:f_1=0$ be a reduced curve of degree $d_1$ and $L$ be a line.
Assume  that all singularities of $C_1$ and $C=C_1 \cup L $ are quasihomogeneous. Then there is an exact sequence  for any integer $k$ given by
$$ 0 \to D_0(f_1)_{k-1}  \longrightarrow  D_0(f)_k \to H^0(L,\OO_{L}(k+1-|R|) )\to  N(f_1)_{k+d_1-2} ,$$
where $R=C_1 \cap L$ as a set.
\end{thm}

Our first result in this paper is the following.
\begin{thm}
\label{thmR}
Let $\A:f=0$ be a line arrangement with $d\geq 6$ and $m\geq 3$. Then
the following hold.

\begin{enumerate}

\item One has $\dim D_0(f)_{d-m} \geq n_m$.

\item  If $\dim D_0(f)_{d-m}  \leq 3$, then $r(\A)=d-m.$

\item If $\dim D_0(f)_{d-m}\leq \binom{k+2}{2}$ for some integer $k\geq 1$, then
\[
        r(\A)\geq d-m-k+1.
\]
In addition, when $m=3$ and $k=1$, one has
\[
        1\leq \dim D_0(f)_{d-3}=n_3\leq 3,
\]
and the first exponents of $\A$ are
\[
        e_1=\cdots=e_{n_3}=d-3.
\]

\end{enumerate}

\end{thm}

For $m=3$, Theorem~\ref{thmB}(3) below shows that the threshold $n_3\leq 3$ in part (2) cannot be increased uniformly. Indeed, in the case $(d,m,n_3)=(7,3,4)$, both $r=3=d-m-1$ and $r=4=d-m$ occur. In particular, the inequality $n_3\leq 3$ is not a necessary condition for $r=d-m$.
\proof
We use the notation from \cite{Der}, in particular for a multiple point 
$$p=(u:v:w) \in \A$$
we denote by $f_{2p}$  the product of the factors in $f$ not vanishing at $p$,
$$D_p=u\partial_x+v\partial_y+w\partial_z$$
and
$$\tilde D_p=f_{2p}D_p-\frac{D_p(f_{2p})}{d}E,$$
where $E=x\partial_x+y\partial_y+z\partial_z$ is the Euler derivation.
If $p$ is a point of multiplicity $m$, then $\tilde D_p \in D_0(f)_{d-m}$ and it satisfies $\tilde D_p(p) \ne 0$ and $\tilde D_p(q)=0$ for any point $q\ne p$ in $\A$ of multiplicity at least 3.
Using \cite[Theorem 1.3]{Der} it follows that $\dim D_0(f)_{d-m}\geq n_m$, since the derivations $\{\tilde D_p\}_p$ are linearly independent, with $p$ running through the set of points of $\A$ of multiplicity $m$.
This proves the claim (1).

To prove (2), assume  that $\dim D_0(f)_{d-m}  \leq 3$
and that
there is $D \in D_0(f)_{d-m-1}$, $D\ne 0$. Then clearly $xD,yD,zD$ are linear independent in $D_0(f)_{d-m}$. This is already a contradiction when $\dim D_0(f)_{d-m} \leq 2$, and hence the proof is complete in this case. 
When $\dim D_0(f)_{d-m}=3$, it follows that for any point $p$ of multiplicity $m$ one has 
$$\tilde D_p \in D_0(f)_{d-m}=S_1 \cdot D.$$
This is a contradiction, since any derivation $\tilde D_p$ is primitive,
i.e. not the multiple of a strictly lower degree derivation, see \cite[proof of Theorem 1.2]{Dca}.

To prove (3), put $n=d-m$ and suppose that there exists
\[
        0\neq D\in D_0(f)_{n-k}.
\]
Then $S_kD\subset D_0(f)_n$ and
\[
        \dim S_kD=\dim S_k=\binom{k+2}{2}.
\]
Under the hypothesis of (3) we therefore have
\[
        D_0(f)_n=S_kD.
\]
For any point $p$ of multiplicity $m$, the local derivation $\widetilde D_p\in D_0(f)_n$ would then be of the form $hD$ with $h\in S_k$. Since $k>0$, this contradicts the primitivity of $\widetilde D_p$. Hence
\[
        D_0(f)_{n-k}=0.
\]
If a nonzero syzygy existed in any lower degree, multiplying it by a suitable nonzero homogeneous polynomial would give a nonzero element of $D_0(f)_{n-k}$, again a contradiction. Thus
\[
        r\geq n-k+1=d-m-k+1.
\]

For $m=3$ and $k=1$, the equality $\dim D_0(f)_{d-3}=n_3$ holds by \cite[Theorem 1.4]{Der}. Hence $n_3\leq 3$ implies $r\geq d-3$, while \eqref{eqR} gives $r\leq d-3$. Therefore $r=d-3$, and the equality $\dim D_0(f)_{d-3}=n_3$ shows that the first $n_3$ minimal generators all have degree $d-3$.
\endproof

Thus, for $d<9$, after excluding the cases $2m>d$ and $m=2$, the only remaining possibilities are
\[
        (d,m)=(6,3),\quad (7,3),\quad (8,3),\quad (8,4).
\]
These four cases are treated in Theorems~\ref{thmA}, \ref{thmB}, \ref{thmC}, and \ref{thmD}, respectively.

In order to study Ziegler pairs of lines in the last sections of this paper, we shall also use the following elementary criterion.

\begin{lemma}
Let \(m\geq 0\).  Consider the multiplication map
\[
\mu_m:S_1\otimes AR(f)_m\longrightarrow AR(f)_{m+1}.
\]
Then the quotient
\[
AR(f)_{m+1}/S_1AR(f)_m
\]
measures the new minimal generators of \(AR(f)\) in degree \(m+1\). 

In particular, their number is
\[
\dim AR(f)_{m+1}-\operatorname{rank}(\mu_m).
\]
\end{lemma}

\begin{proof}
The image of \(\mu_m\) consists exactly of the degree \(m+1\) syzygies obtained by multiplying degree \(m\) syzygies by linear forms.  Hence the quotient
\[
AR(f)_{m+1}/S_1AR(f)_m
\]
is the space of degree \(m+1\) syzygies modulo those generated in lower degree.  Its dimension is therefore the number of new minimal generators in degree \(m+1\).  Since
\[
\dim \operatorname{Im}(\mu_m)=\operatorname{rank}(\mu_m),
\]
the formula follows.
\end{proof}
Finally we recall the following result due to  du Plessis and Wall, see \cite[Theorem 3.2]{duPCTC}.
\begin{thm}
\label{thmCTC}
For positive integers $d$ and $r$, define two new integers by 
$$\tau(d,r)_{min}=(d-1)(d-r-1)  \text{ and } 
\tau(d,r)_{max}= (d-1)^2-r(d-r-1).$$ 
Then, if $C:f=0$ is a reduced curve of degree $d$ in $\PP^2$ and  $r={\rm mdr}(f)$,  one has
$$\tau(d,r)_{min} \leq \tau(C) \leq \tau(d,r)_{max}.$$
Moreover, for $r={\rm mdr}(f) \geq d/2$, the stronger inequality
$$\tau(C) \leq \tau(d,r)_{max}':=\tau(d,r)_{max} - \binom{2r+2-d}{2}$$
holds.

\end{thm}
With this notation, 
\begin{equation} 
\label{eqFC}
\tau(C)= \tau(d,r)_{max}
\end{equation} 
 if and only if
$2r <d$ and $C$ is a free curve, see \cite{duPCTC,DimR}. On the other hand if
$2r\geq d$ and
\begin{equation} 
\label{eqMTC}
 \tau(C)= \tau(d,r)_{max}' 
 \end{equation} 
 then by definition $C$ is a {\it Tjurina maximal curve}, see \cite{maxT}.

\section{Line arrangements of \texorpdfstring{$d\leq 7$}{} lines}

In this section we determine ${\rm mdr}$ in the remaining cases with $d\leq 7$ and record the corresponding exponent sequences. For the various types of curves that occur in these statements we refer the reader to \cite{Brian, NH}.
\bigskip

\begin{thm}
\label{thmA}
Let $\A$ be a line arrangement such that $(d,m)=(6,3)$. Then the following cases are possible.

\begin{enumerate}

\item $\tau(\A)=19$, $n_3=4$ and then $\A$ is free with exponents $(2,3)$.

\item $\tau(\A)=18$, $n_3=3$ and then $\A$ is nearly free with exponents $(3,3,3)$.

\item $\tau(\A)=17$, $n_3=2$ and then $\A$ is minimal POG with exponents $(3,3,4)$.

\item $\tau(\A)=16$, $n_3=1$ and $\A$ is obtained from a nodal arrangement $\B$ of 5 lines by adding a line $L$ passing through exactly one node of $\B$. In this case $\A$ is a curve of type $2B$ with exponents $(3,4,4,4)$.

\end{enumerate}
\end{thm}
 
\proof 
The inequality \eqref{eqMS} yields in this case $r \geq 2$.
It follows from 
\cite{Dca} that either $r=m-1=2$ and $\A=\A'$ is free with exponents $(2,3)$,
or $\A=\A''$ has $r=d-m=3$.
In the first situation we get
$$\tau(\A')=25-6=19,$$
while in the second one we get
$$\tau(\A'')\leq 25-6-1=18.$$
Consider the system of equations, where $n_k$ denotes the number of points in $\A$ of multiplicity $k$.
$$n_2+3n_3=\binom{d}{2}=15 \text{ and } n_2+4n_3=\tau(\A).$$
It follows that $\A'$ has $n_3=4$ and $\A''$ has $1 \leq n_3 \leq 3$.
In conclusion, the triple $(d,m,\tau(\A))$ or the triple $(d,m,n_3)$ determines $r$ in this case.
Note that if $n_3 \geq 2$ in the case $\A''$, it follows that
$$\tau(\A'') \geq 10+3+3+1=17,$$
see \cite[Theorem 1.1 (i)]{minT}.
Similarly, when $n_3=1$, we have
$$\tau(\A'') \geq 10+3+2+1=16,$$
see \cite[Theorem 1.1 (ii)]{minT}.
When $\tau(\A'')=18$, then $\A''$ is nearly free with exponents $(3,3,3)$,
and when $\tau(\A'')=17$, then $\A''$ is a minimal POG arrangement with exponents $(3,3,4)$, see \cite[Theorem 1.5]{Brian}.
To get information on the exponents of $\A$ in case (4), we apply Theorem \ref{thm2} with $C_1=\B: f_1=0$ and $k=3$ and get
$$ 0 = D_0(f_1)_{2}  \longrightarrow  D_0(f)_3 \to H^0(L,\OO_{L}) =\C,$$
since $|R|=4$. It follows that $e_1=3$ and $e_2>3$.
If we take now $k=4$ we get
$$ 0 \to D_0(f_1)_{3}=\C^4  \longrightarrow  D_0(f)_4 \to H^0(L,\OO_{L}(1))=\C^2 \to N(f_1)_7=0,$$
where the last vanishing follows from $\dim N(f_1)_7 =\dim N(f_1)_2=0$,
see \cite[Theorem 3.2]{Port}.
Hence $e_2=e_3=e_4=4$. It follows that the type of $\A$ is
$$t(\A)=e_1+e_2-d+1=2,$$
see \cite{NH}. Using the classification of type 2 curves given in
\cite[Proposition 1.11]{NH} we infer that 
$\A$ is a curve of type $2B$.
\endproof
\begin{rk}
Observe that for $(d,m)=(6,3)$ the case $n_{3}=5$ is not possible due to our upper-bound on the total Tjurina number. Indeed, since
\[
n_2+3n_3=15,\qquad \tau(\mathcal A)=n_2+4n_3=15+n_3,
\]
the case \(n_3=5\) would give \(\tau(\mathcal A)=20\).
This is impossible, because for a line arrangement of six lines one has
\[
\tau(\mathcal A)\le (d-1)^2-r(d-1-r)\le 25-6=19.
\]
Hence \(n_3\le 4\).
\end{rk}

\begin{thm}
\label{thmB}
Let $\A$ be a line arrangement such that $(d,m)=(7,3)$. Then the following cases are possible.

\begin{enumerate}

\item $\tau(\A)=27$, $n_3=6$ and then $\A$ is  free with exponents $(3,3)$.

\item $\tau(\A)=26$, $n_3=5$ and then $\A$ is nearly free with exponents $(3,4,4)$.

\item $\tau(\A)=25$, $n_3=4$ and then the following two cases occur.

\begin{enumerate}
\item If there is a line $L_0 \in \A$ containing no triple point, then $\A$ is a plus-one generated curve with exponents $(3,4,5)$.

\item If there is a line $L_1 \in \A$ containing exactly one triple point, then $r=4$ and $\A$ is a maximal Tjurina curve of type $(7,4)$.

\end{enumerate}

\item $\tau(\A)=24$, $n_3=3$, then $e_1=e_2=e_3=4$ and the following cases may occur.

\begin{enumerate}
\item If there is a line $L_0 \in \A$ containing no triple point, then $\A$ has type 2B with exponents $(4,4,4,5)$.

\item If there is no line $L_0 \in \A$ containing no triple point, then  $\A$ has type 2A with exponents $(4,4,4)$.

\end{enumerate}

\item $\tau(\A)=23$, $n_3=2$, then $\A$ has type 2B with exponents $(4,4,5,5)$.

\item $\tau(\A)=22$, $n_3=1$, then $\A$ has type 3C with exponents
$(4,5,5,5,5)$.

\end{enumerate}
\end{thm}

\proof

The inequality \eqref{eqMS} yields in this case $r \geq 3$.
On the other hand, the results in \cite{Dca} tell
us that
$r \leq  d-m=4.$

If $r=3$ then we know that
$$18=(d-1)(d-1-r) \leq \tau(\A) \leq 36-9=27.$$
In this case $\tau(\A)=27$ if and only if $\A$ is free with exponents $(3,3)$ and $\tau(\A)=26$ if and only if $\A$ is nearly free, hence it has exponents $(3,4,4)$.
The case $\tau(\A)=25$ corresponds to $\A$ being a minimal POG arrangement, hence the exponents are $(3,4,5)$ in this last subcase,
see \cite[Theorem 1.5]{Brian}. 

In fact we have in this case
$$\tau(\A) \geq 18+3+3+1=25,$$
see \cite[Theorem 1.1 (ii)]{minT}.
And if $r=4$, then
$$12=(d-1)(d-1-r) \leq \tau(\A) \leq 36-8-3=25.$$
Hence only the case $\tau(\A)=25$ may cause problems in deciding the value of $r\in \{3,4\}$. 

Consider the system of equations
$$n_2+3n_3=\binom{d}{2}=21 \text{ and } n_2+4n_3=\tau(\A)=25.$$
It follows that $n_2=9$ and $n_3=4$ when $d=7$ and $\tau(\A)=25$.
If every line of \(\mathcal A\) contained at least two triple points, then counting incidences between lines and triple points would give
\[
3n_3\ge 2d=14,
\]
hence \(n_3\ge 5\), contradicting \(n_3=4\).

We record that the two possibilities which occur below are in fact
mutually exclusive. Suppose first that there is a line
$L_0\in\A$ containing no triple point. Every other line contains at
most two triple points. Indeed, if a line $L\neq L_0$ contained three
triple points, then each of these three points would require two
further arrangement lines through it. Since two distinct points of
$L$ cannot lie on the same further arrangement line, this would use
all six lines distinct from $L$, including $L_0$, contradicting the
choice of $L_0$. On the other hand,
\[
        \sum_{L\in\A} n_3(L)=3n_3=12.
\]
Thus the six lines different from $L_0$ contain exactly two triple
points each. In particular, no line contains exactly one triple
point.

Conversely, if there is no line containing zero triple points, then
every line contains at least one triple point. Since
\[
        \sum_{L\in\A} n_3(L)=12<14=2d,
\]
at least one line contains exactly one triple point. Hence precisely
one of the two alternatives in part~(3) occurs. It follows that there is either a line $L_0 \in \A$ without triple points, or a line $L_1 \in \A$ containing exactly one triple point.
Let $\B_0$ (resp. $\B_1$) be the line arrangement obtained from $\A$ by deleting $L_0$ (resp. $L_1$). Using Bezout Theorem, we see that the intersection $R_0=\B_0 \cap L_0$ (resp. $R_1=\B_1 \cap L_1$) consists of 6 (resp. 5) points.

The Tjurina number $\tau(B_0)$ is obtained from $\tau(\A)$ by substracting 
6, since the 6 nodes of $\A$ on $L_0$ are now simple points in $\B_0$. Hence $\tau(\B_0)=25-6=19$ and $r_0=r(\B_0)\geq r-1\geq 2$.  
 Hence, by Theorem \ref{thmA} (1),  we get  $r_0=2$ and it follows that $r=3$ if such a line $L_0$ exists. Since $\B_0$ is free in this case, it follows from \cite{POG,MP} that $\A$ is a plus-one generated curve with the given exponents.

Consider now the case when a line $L_1$ exists. The Tjurina number $\tau(B_1)$ is obtained from $\tau(\A)$ by substracting 
7, since the 4 nodes of $\A$ on $L_1$ are now simple points in $\B_1$
and the triple point of $\A$ on $L_1$ becomes a node in $\B_1$. Hence $\tau(B_1)=25-7=18$.  
 Hence, by Theorem \ref{thmA} (2),  we get  $r_1=r(\B_1)=3$.
To  apply Theorem \ref{thm2} to the case at hand, we take $C_1=\B_1:f_1=0$, $k=3$ and get
\[
0=D_0(f_1)_2\longrightarrow D_0(f)_3
\longrightarrow H^0\bigl(L_1,\OO_{L_1}(4-|R_1|)\bigr)
=H^0\bigl(L_1,\OO_{L_1}(-1)\bigr)=0,
\]
since $|R_1|=5$.
It follows that $r=4$ when there is a line $L_1$.

It remains to consider the cases $n_3\leq 3$. Since
\[
        n_2+3n_3=21,
        \qquad
        \tau(\A)=n_2+4n_3,
\]
we have
\[
        \tau(\A)=21+n_3.
\]

Suppose first that $n_3=3$. Then $\tau(\A)=24$. By
Theorem~\ref{thmR},
\[
        r=4,\qquad
        \dim D_0(f)_4=3,
\]
and the three minimal generators in degree $4$ are the local
derivations associated with the three triple points. In particular,
the first three exponents are
\[
        e_1=e_2=e_3=4.
\]
Since
\[
        e_1+e_2-d+1=4+4-7+1=2,
\]
the arrangement has type $2$. The classification of type $2$ curves,
see \cite[Proposition 1.11]{NH}, leaves precisely the two degree
sequences
\[
        (4,4,4)
        \qquad\text{and}\qquad
        (4,4,4,5).
\]

Assume first that there is a line $L_0\in\A$ containing no triple
point. Then
\[
        |L_0\cap(\A\setminus L_0)|=6.
\]
The addition--deletion estimate
\cite[Corollary 1.9(2)]{Der} implies that the maximal exponent is at
least $5$. Hence the sequence $(4,4,4)$ is impossible, and therefore
\[
        \operatorname{Exp}(\A)=(4,4,4,5).
\]
Thus $\A$ is of type $2B$.

Assume now that no line of $\A$ contains zero triple points. Let
$p_1,p_2,p_3$ be the three triple points. Suppose first that they are
not collinear. If all three lines determined by the pairs
$p_i,p_j$ belonged to $\A$, then these three lines, together with the
three remaining lines needed to make each $p_i$ a triple point, would
account for only six lines containing triple points. The seventh line
would contain no triple point, a contradiction. Hence at most two of
the three lines determined by $p_1,p_2,p_3$ belong to $\A$.
By \cite[Theorem 3.3]{Der} one therefore has the direct sum
\[
S_1\widetilde D_{p_1}\oplus
S_1\widetilde D_{p_2}\oplus
S_1\widetilde D_{p_3}
\subset D_0(f)_5 .
\]
Its dimension is $9$. On the other hand,
\cite[Theorem 1.1]{Der} gives
\[
\dim D_0(f)_5
=
24+3\binom{7}{2}-\binom{13}{2}
=
9.
\]
Hence
\[
D_0(f)_5=
S_1D_0(f)_4,
\]
so there is no new minimal generator in degree $5$. Thus
\[
        \operatorname{Exp}(\A)=(4,4,4).
\]

It remains only to consider the possibility that $p_1,p_2,p_3$ are
collinear. Their common line must then belong to $\A$; otherwise an
arrangement line could contain at most one of the three triple
points, and the nine line--triple incidences would require at least
nine distinct lines. Choose coordinates such that the common line is
\[
        H:z=0
\]
and
\[
        p_1=(1:0:0),\qquad
        p_2=(0:1:0),\qquad
        p_3=(1:1:0).
\]
The six remaining arrangement lines occur in three pairs through
$p_1,p_2,p_3$, respectively. Up to multiplication by nonzero
constants, their equations can be written as
\[
y+\alpha_i z,\qquad
x+\beta_i z,\qquad
x-y+\gamma_i z,
\qquad i=1,2.
\]
Consequently, after rescaling the corresponding local derivations,
\[
\left.f_{2p_1}\right|_H=x^2(x-y)^2,\qquad
\left.f_{2p_2}\right|_H=y^2(x-y)^2,\qquad
\left.f_{2p_3}\right|_H=x^2y^2.
\]
Consider a relation
\[
P_1\widetilde D_{p_1}
+
P_2\widetilde D_{p_2}
+
P_3\widetilde D_{p_3}=0,
\qquad P_i\in S_1.
\]
Evaluating at the three triple points and using
\cite[Theorem 1.3]{Der} gives
\[
P_1(p_1)=P_2(p_2)=P_3(p_3)=0.
\]
Hence, on $H$,
\[
P_1=ay,\qquad P_2=bx,\qquad P_3=c(x-y)
\]
for some $a,b,c\in\C$. Restricting the relation to $H$ and using the
formula for the local derivations from \cite[Theorem 1.3]{Der}, its
first two components, up to nonzero common factors, are
\[
(3a-2b-2c)x+(-5a+4b+5c)y
\]
and
\[
(4a-5b-5c)x+(-2a+3b+2c)y.
\]
Thus
\[
\begin{aligned}
3a-2b-2c&=0,\\
-5a+4b+5c&=0,\\
4a-5b-5c&=0,\\
-2a+3b+2c&=0.
\end{aligned}
\]
This system has only the solution $a=b=c=0$. Therefore
$P_i|_H=0$ for all $i$, so each $P_i$ is divisible by $z$. Dividing
the original relation by $z$ would give a constant linear relation
among
\[
        \widetilde D_{p_1},
        \widetilde D_{p_2},
        \widetilde D_{p_3},
\]
which is impossible by \cite[Theorem 1.4]{Der}. Hence also in the
collinear case
\[
D_0(f)_5=S_1D_0(f)_4
\]
and
\[
        \operatorname{Exp}(\A)=(4,4,4).
\]
This completes the proof of part~(4).

For $n_3=2$, one has $\tau(\A)=23$. Theorem~\ref{thmR} gives
\[
        r=4,\qquad e_1=e_2=4.
\]
Hence $\A$ is again a curve of type $2$, and the type $2$
classification \cite[Proposition 1.11]{NH}, together with
$\tau(\A)=23$, gives uniquely
\[
        \operatorname{Exp}(\A)=(4,4,5,5).
\]
Thus $\A$ has type $2B$, proving part~(5).

Finally, let $n_3=1$. Then $\tau(\A)=22$, and
Theorem~\ref{thmR} gives
\[
        r=4,\qquad \dim D_0(f)_4=1.
\]
Moreover, \cite[Theorem 1.1]{Der} gives
\[
\dim D_0(f)_5
=
22+3\binom{7}{2}-\binom{13}{2}
=
7.
\]
The multiples of the unique degree $4$ generator span a
$3$-dimensional subspace of $D_0(f)_5$. Hence there are exactly
\[
        7-3=4
\]
new minimal generators in degree $5$. Since for an essential
arrangement of seven lines every exponent is at most
$d-2=5$, this determines the complete degree sequence:
\[
        \operatorname{Exp}(\A)=(4,5,5,5,5).
\]
This is the type $3C$ case, and proves part~(6).
\endproof
\begin{rk}
In the setting of $(d,m)=(7,3)$, the value \(n_3=7\) would imply \(n_2=0\), hence the intersection lattice would be the Fano plane. Since the Fano plane is not realizable over a field of characteristic zero, this case cannot occur for complex line arrangements.
\end{rk}
The following examples illustrate the alternatives appearing in Theorem 3.3.
\begin{ex}
\label{exB}
(i) The line arrangement
$$\A: f=xyz(x-z)(y-z)(x-y)(x+y-3z)=0$$
is as in (3) (a) above with $L_0:x+y-3z=0$ and has exponents $(3,4,5)$.
On the other hand, the arrangement
$$\A: f=xyz(x-z)(y-z)(x+y-2z)(x-7y)=0$$
is as in (3) (b) with $L_1:x-7y=0$ and exponents $(4,4,4,4)$.

(ii) The line arrangement
$$\A: f= xyz(x+y)(x+z)(y+z)(x+2y+5z)=0$$
is as in (4) (a) above with $L_0:x+2y+5z=0$ and has exponents $(4,4,4,5)$.

The line arrangement
$$\A: f= (x-y)(x-z)(y-z)(x+z)(x+2z)(y+3z)(y+5z) =0$$
is as in (4) (b) above  and has exponents $(4,4,4)$.

\end{ex}

\section{Line arrangements of \texorpdfstring{$8$}{8} lines} 
The multiplicities $m$ that have to be discussed in this case are $3$ and $4$. We start with multiplicity $4$.

\begin{thm}
\label{thmC}
Let $\A$ be a line arrangement such that $(d,m)=(8,4)$. Then the following cases are possible.

\begin{enumerate}

\item If $\tau(\A)=37$, then $\A$ is free with exponents $(3,4)$.

\item If $\tau(\A)=36$,  then $\A$ is nearly free with exponents $(4,4,4)$.

\item If $\tau(\A)=35$,  then $\A$ is minimal POG with exponents $(4,4,5)$.

\item If $\tau(\A)=34$, then $r=4$ and the numerical possibilities are $(n_3,n_4)=(0,2)$ and $(n_3,n_4)=(3,1)$. More precisely:
\begin{enumerate}
\item If $(n_3,n_4)=(0,2)$, then $\A$ is a POG arrangement with exponents $(4,4,6)$.
\item If $(n_3,n_4)=(3,1)$, then two exponent sequences may occur: either $\A$ is a POG arrangement with exponents $(4,4,6)$, or $\A$ has type $2B$ with exponents $(4,5,5,5)$.
\end{enumerate}

\item If $\tau(\A)=33$, then $(n_3,n_4)=(2,1)$, $r=4$, and the following two cases may occur.
\begin{enumerate}
\item $\A$ has type 2A with exponents $(4,5,5)$.
\item $\A$ has type 2B with exponents $(4,5,5,6)$.
\end{enumerate}

\item If $\tau(\A)=32$,  then $\A$ has type 2B with exponents $(4,5,6,6)$.

\item If $\tau(\A)=31$,  then $\A$ has type 3C with exponents $(4,6,6,6,6)$.

\end{enumerate}
\end{thm}

\begin{rk}
\label{rkC}
(i) Even when the numerical type of the arrangement $\A$, namely the sequence $(n_2,n_3,n_4)$, is fixed in Theorem~\ref{thmC}, there may be several possibilities for the intersection lattice $L(\A)$. For instance, in case (4), when $(n_2,n_3,n_4)=(16,0,2)$, there are at least two different intersection lattices. One is obtained by taking two points in $\PP^2$ and four otherwise generic lines through each of them; see \cite[Proposition 1.17]{NH}. The other is obtained by taking two points, joining them by one arrangement line, and adding three further lines through each point. This gives an arrangement $\B$ of seven lines with two quadruple points and nine nodes. Adding an eighth line $L$ which avoids the singular points of $\B$ creates seven new nodes and gives again $(n_2,n_3,n_4)=(16,0,2)$. The two intersection lattices are distinct.

(ii) One always has $n_4\leq 2$. Indeed, for a quadruple point $p$, let $I_p$ be the set of the four arrangement lines through $p$. If $p\neq q$, then
\[
        |I_p\cap I_q|\leq 1,
\]
since two distinct lines cannot meet at two distinct points. Suppose that there are three quadruple points $p,q,r$. If no arrangement line passes through all three, then
\[
        |I_p\cup I_q\cup I_r|\geq 12-3=9.
\]
If one arrangement line passes through all three, then the three pairwise intersections coincide with that line and
\[
        |I_p\cup I_q\cup I_r|\geq 12-3+1=10.
\]
Both inequalities contradict $d=8$. Hence $n_4\leq 2$.

(iii) The combinatorial identities
\[
        n_2+3n_3+6n_4=28,
        \qquad
        \tau(\A)=n_2+4n_3+9n_4
\]
give
\[
        n_3+3n_4=\tau(\A)-28.
\]
Since $m=4$, one has $n_4\geq 1$, while part (ii) gives $n_4\leq 2$. Thus, for $\tau=37,36,35,34$, the numerically allowed pairs $(n_3,n_4)$ are respectively
\[
\begin{array}{c|c}
\tau & \text{possible }(n_3,n_4)\\
\hline
37 & (6,1)\text{ or }(3,2)\\
36 & (5,1)\text{ or }(2,2)\\
35 & (4,1)\text{ or }(1,2)\\
34 & (3,1)\text{ or }(0,2)
\end{array}.
\]
Within the second construction in part (i), one has $n_4=2$. If the added line $L$ is chosen through $t\in\{0,1,2,3\}$ suitable nodes of $\B$ and otherwise avoids its singular locus, then $n_3=t$. Hence, within this particular construction, cases (1), (2), and (3) of Theorem~\ref{thmC} are realized by $t=3,2,1$, respectively, while the $(n_3,n_4)=(0,2)$ branch of case (4) is obtained for $t=0$. This correspondence is specific to that construction and is not a classification of all numerical types occurring in the theorem.
\end{rk}

\proof
The inequality \eqref{eqMS} yields in this case $r \geq 2$.
Since $m=4$, we have using \eqref{eqR} that $3 =m-1 \leq r \leq 4= d-m$. It follows that
$r=3$ if and only if
$$\tau(\A)=49-12=37$$
and in this case $\A$ is free with exponents $(3,4)$, see for instance \cite[Corollary 1.3]{Brian}.
When $r=4$, then
$$\tau(\A)\leq 37-1=36$$
and the equality holds if and only if $\A$ is nearly free with exponents
$(4,4,4)$, see for instance \cite[Corollary 1.4]{Brian}. In fact, in this case we have
$$n_2+3n_3+6n_4=\binom{d}{2}=28 \text{ and } n_2+4n_3+9n_4=\tau(\A).$$
It follows that
$$n_3+3n_4=\tau(\A)-28=36-28=8.$$
Hence in this case we have $n_4\leq 2$ as explained in Remark \ref{rkC} (ii), and also
$$3=\dim D_0(f)_4>2 \geq n_4.$$
When $\tau(\A)=35$, then $\A$ is minimal POG, and hence has exponents $(4,4,5)$, see  \cite[Theorem 1.5]{Brian}.

When $\tau(\A)=34$, either $(n_3,n_4)=(0,2)$ or $(n_3,n_4)=(3,1)$. In the first case
\[
        \dim D_0(f)_4\geq n_4=2,
\]
so $e_1=e_2=4$. The type $1$ classification then gives the POG sequence $(4,4,6)$.

In the second case the weak numerical data do not determine the full exponent sequence. Since $r=4$, the type $1$ and type $2$ classifications, see \cite{Brian,NH}, leave precisely the two possibilities
\[
        (4,4,6)
        \qquad\text{and}\qquad
        (4,5,5,5).
\]
Both possibilities are realized; see Example~\ref{exC34} below.

For \(\tau=33\) and \(\tau=32\), the numerical types are respectively
\[
(n_3,n_4)=(2,1),\qquad (n_3,n_4)=(1,1).
\]
Since \(r=4\), the first exponent is \(4\). For \(\tau=33\), the formulas for curves of type $2$ allow both the type 2A sequence
\[
(4,5,5)
\]
and the type 2B sequence
\[
(4,5,5,6).
\]
Thus the numerical data $(d,m,\tau)=(8,4,33)$ do not determine the full degree sequence, although they determine ${\rm mdr}=4$. For \(\tau=32\), the type $2$ formulas give the sequence \((4,5,6,6)\).

Finally, when $\tau(\A)=31,$ we have  $n_4=1$, $n_3=0$ and $n_2=22$.
In this last case, $\A$ is obtained from an arrangement $\B:f_1=0$ of 7 lines having exactly one  point $p$ of multiplicity 4 besides double points, by adding a new generic line $L$.  It is easy to see that $\B$ is a type 2B arrangement with exponents $(3,5,5,5)$ and $\tau(\B)=24$. It follows that the first exponents of $\A$ are $e_1=4$, $e_2=e_3=e_4=6$ and $e_5 \geq 6$, see \cite[Corollary 4.2]{Der}.
In other words, $\A$ is an arrangement of type 3C
 and the formula in \cite[Theorem 3.1 (4)]{type3} yields
$$\tau(\A)=16+24+36-12-18-6-6-e_5+3=37-e_5=31,$$
and hence $e_5=6$ as well.

\endproof

\begin{ex}
\label{exC34}
Both alternatives in Theorem~\ref{thmC}(4)(b) occur. Consider
\[
F_{\triangle}=xyz(x+y+z)(2x-y)(3x-z)(3y-2z)(12x-3y-2z).
\]
The arrangement $F_{\triangle}=0$ has
\[
        (n_2,n_3,n_4)=(13,3,1),
        \qquad
        \tau=34.
\]
An exact computation gives
\[
(\dim AR(F_{\triangle})_3,\dim AR(F_{\triangle})_4,
 \dim AR(F_{\triangle})_5,\dim AR(F_{\triangle})_6)
=(0,2,6,13),
\]
with
\[
        \rank(\mu_4)=6,
        \qquad
        \rank(\mu_5)=12.
\]
Hence
\[
        \operatorname{Exp}(F_{\triangle})=(4,4,6).
\]

On the other hand, let
\[
F_{\star}=xyz(x+y+z)(2x-y)(3x-z)(5x-y-z)(12x-3y-2z).
\]
Again
\[
        (n_2,n_3,n_4)=(13,3,1),
        \qquad
        \tau=34,
\]
but
\[
(\dim AR(F_{\star})_3,\dim AR(F_{\star})_4,
 \dim AR(F_{\star})_5,\dim AR(F_{\star})_6)
=(0,1,6,13),
\]
and
\[
        \rank(\mu_4)=3,
        \qquad
        \rank(\mu_5)=13.
\]
Therefore
\[
        \operatorname{Exp}(F_{\star})=(4,5,5,5).
\]
Thus the weak combinatorics $(n_2,n_3,n_4)=(13,3,1)$ does not determine the full exponent sequence. Obviously the arrangements do not have isomorphic intersection lattices.
\end{ex}

\begin{ex}
The type 2A alternative in Theorem~\ref{thmC} for $\tau=33$ is realized by
\[
f=z(3x-2z)(x-y-z)(2x-y-2z)(3x-3y-2z)(3x-3y+2z)(3x+2y-3z)(3x+3y-z).
\]
Numbering the factors in the displayed order, the maximal multiple sets are
\[
\{1,3,5,6\},\qquad \{2,3,8\},\qquad \{3,4,7\}.
\]
Hence $(n_2,n_3,n_4)=(16,2,1)$ and $\tau(\A)=33$. An exact computation gives
\[
(\dim AR(f)_3,\dim AR(f)_4,\dim AR(f)_5,\dim AR(f)_6)=(0,1,5,12),
\]
with $\rank(\mu_4)=3$ and $\rank(\mu_5)=12$. Therefore
\[
\operatorname{Exp}(f)=(4,5,5).
\]
\end{ex}

Finally we consider the last remaining case $m=3$.

\begin{thm}
\label{thmD}
Let $\A$ be a line arrangement such that $(d,m)=(8,3)$. Then the following cases are possible.

\begin{enumerate}

\item If $\tau(\A)=36$, then $r=4$, $n_3=8$, and $\A$ is nearly free with exponents $(4,4,4)$.

\item If $\tau(\A)=35$, then $r=4$, $n_3=7$, and $\A$ is minimal POG with exponents $(4,4,5)$.

\item If $\tau(\A)=34$, then $r=4$, $n_3=6$, and the following two cases may occur.
\begin{enumerate}
\item $\A$ is POG with exponents $(4,4,6)$.
\item $\A$ has type 2B with exponents $(4,5,5,5)$.
\end{enumerate}

\item If $\tau(\A)=33$, then $n_3=5$, and the following cases may occur.
\begin{enumerate}
\item $\A$ has type 2A with exponents $(4,5,5)$ and $r=4$.
\item $\A$ has type 2B with exponents $(4,5,5,6)$ and $r=4$.
\item $\A$ has type 3C with exponents $(5,5,5,5,5)$ and $r=5$.
\end{enumerate}

\item If $\tau(\A)=32$, then  $n_3=4$ and  the following cases may occur.

\begin{enumerate}

\item $\A$ has type 2B with exponents $(4,5,6,6)$.

\item $\A$ has type 3B' with exponents $(5,5,5,5)$.

\item $\A$ has type 3C with exponents $(5,5,5,5,6)$.

\end{enumerate}

\item If $\tau(\A)=31$, then $r=5$, $n_3=3$, and  the following cases may occur.

\begin{enumerate}

\item $\A$ has type 3B with exponents $(5,5,5,6)$.

\item $\A$ has type 3C with exponents $(5,5,5,6,6)$.

\end{enumerate}

\item If $\tau(\A)=30$, then $r=5$, $n_3=2$, and $\A$ has type 3C with exponents $(5,5,6,6,6)$.

\item If $\tau(\A)=29$, then $r=5$, $n_3=1$, and $\A$ has type 4 with exponents $(5,6,6,6,6,6)$.

\end{enumerate}
\end{thm}

\begin{rk}
\label{rkD}
The cases in Theorem~\ref{thmD} where the same weak numerical data allow different values of ${\rm mdr}$ can be distinguished by the intersection lattice.

Consider first case (4), where $n_3=5$. If every line of $\A$ contained at least two triple points, then counting incidences would give
\[
3n_3\geq 2d=16,
\]
a contradiction. Hence there is a line containing either no triple point or exactly one triple point. If there is a line $L_0\in\A$ containing only double points, then \cite[Corollary 1.9(2)]{Der} gives a minimal generator in degree at least $6$. Among the alternatives in (4), this forces case (b), with exponents $(4,5,5,6)$ and $r=4$.

Assume now that no such line $L_0$ exists. Then there is a line $L_1$ containing exactly one triple point. Let $\B_1=\A\setminus L_1$; then $|L_1\cap\B_1|=6$, and $\B_1$ is as in Theorem~\ref{thmB}(3). The two combinatorial alternatives in Theorem~\ref{thmB}(3), together with
the exact sequence of Theorem~\ref{thm2} for $k=4$, show that the first
alternative yields $r=4$, hence case~(4a) or~(4b), whereas the second
yields $r=5$, hence case~(4c). This is sufficient for determining
$\operatorname{mdr}(\mathcal A)$; we do not need to distinguish between
cases~(4a) and~(4b).

In case (5), let $s$ denote the number of lines of $\A$ containing no
triple point. We claim that
\[
        r(\A)=4 \quad\Longleftrightarrow\quad s\geq 2.
\]
Indeed, assume first that $s\geq 2$, and let $L_0$ be one of the lines
containing only double points. Set
\[
        \B=\A\setminus L_0.
\]
Then $\B$ is an arrangement of seven lines with $n_3=4$, and it still
contains a line having no triple point. Hence, by
Theorem~\ref{thmB}(3)(a), one has
\[
        r(\B)=3.
\]
Since $|L_0\cap\B|=7$, Theorem~\ref{thm2}, applied with $k=4$, gives an
injection
\[
        0\longrightarrow D_0(f_{\B})_3
        \longrightarrow D_0(f_{\A})_4.
\]
Thus $D_0(f_{\A})_4\neq 0$, and therefore $r(\A)=4$.

Assume now that $s=1$, and let $L_0$ be the unique line containing no
triple point. Again put $\B=\A\setminus L_0$. Then $\B$ has seven lines
and $n_3=4$, but no line of $\B$ contains only double points. Hence
Theorem~\ref{thmB}(3)(b) gives
\[
        r(\B)=4.
\]
Since $|L_0\cap\B|=7$, Theorem~\ref{thm2}, with $k=4$, yields
\[
0=D_0(f_{\B})_3
\longrightarrow D_0(f_{\A})_4
\longrightarrow
H^0\bigl(L_0,\mathcal O_{L_0}(-2)\bigr)=0.
\]
Consequently $D_0(f_{\A})_4=0$, and hence $r(\A)=5$.

Finally, suppose that $s=0$. Since
\[
        \sum_{L\in\A} n_3(L)=3n_3=12<16=2d,
\]
there is a line $L_1\in\A$ containing exactly one triple point. Let
$\B_1=\A\setminus L_1$. Then $\B_1$ has seven lines and three triple
points, so Theorem~\ref{thmB}(4) gives
\[
        r(\B_1)=4.
\]
Moreover $|L_1\cap\B_1|=6$. Applying Theorem~\ref{thm2} with $k=4$ we
obtain
\[
0=D_0(f_{\B_1})_3
\longrightarrow D_0(f_{\A})_4
\longrightarrow
H^0\bigl(L_1,\mathcal O_{L_1}(-1)\bigr)=0.
\]
Thus again $r(\A)=5$.

It follows that, in case (5), the value of ${\rm mdr}$ is completely
determined by the intersection lattice: it is equal to $4$ precisely
when at least two lines contain only double points, and it is equal to
$5$ otherwise. We do not need here to distinguish between the two
possible degree sequences in the latter case.

In case (6), both possible degree sequences have $r=5$. We make no combinatorial claim distinguishing these two degree sequences; the number of lines containing only double points does not give such a criterion in general.
\end{rk}

\begin{proof}

The inequality \eqref{eqMS} yields in this case $r \geq 4$. Note that in this case we have
$n_2=28-3n_3$, and hence
$$\tau(\A)=28-3n_3+4n_3=28+n_3\geq 29.$$
For $n_3=1$, by Theorem \ref{thmR} we have $r=d-m=5$, hence the exponents are $(5,6, \ldots, 6)$. If 6 is repeated $s\geq 2$ times, then the formula in
\cite[Theorem 3.1]{typeN} yields
$$29=\tau(\A)=91-11t-6t+R,$$
where $t=e_1+e_2-(d-1)=4$ and 
$$2R=t^2-(\epsilon_1^2+ \ldots +\epsilon_{s-1}^2)$$
with $\epsilon_1+ \ldots +\epsilon_{s-1}=t=4$ and $\epsilon_i >0$ some integers. It follows that
$$\epsilon_1^2+ \ldots +\epsilon_{s-1}^2=4,$$
which yields $s=5$ and $\epsilon_1= \ldots =\epsilon_{4}=1$.
Hence the exponents of $\A$ in this case are $(5,6,6,6,6,6)$, in other words $\A$ is a 6-syzygy curve of type $4T_{\pi_8}$ in the notation from
\cite[Corollary 4.1]{typeN}.

For $n_3=2$, by Theorem \ref{thmR}  the exponents are $(5,5,6, \ldots, 6)$, and hence $\A$ is a curve of type 3. Using \cite[Theorem 3.1]{type3}, it follows that $\A$ has type $3C$ with exponents $(5,5,6,6,6)$.

For \(n_3=3\), Theorem~\ref{thmR} gives us the first three exponents 
\[
        e_1=e_2=e_3=d-3=5.
\]
Moreover, since $m=3<d$, the arrangement is essential; hence the bound $e_s\leq d-2$ applies, and every remaining exponent is equal to $6$. Hence the degree sequence has the form
\[
        (5,5,5,6,\ldots,6).
\]
In particular,
\[
        e_1+e_2-d+1=5+5-8+1=3,
\]
so $\mathcal{A}$ is a curve of type $3$.

We now apply the classification of type $3$ curves from
\cite[Theorem~3.1]{type3}. In the present situation $d=8$,
$\tau(A)=31$, and the first three exponents are $(5,5,5)$, the classification
leaves precisely the following two possibilities. Either $\mathcal{A}$ is of type $3B$, and
then the degree sequence is
\[
        (5,5,5,6),
\]
or $\mathcal{A}$ is of type $3C$, and then the degree sequence is
\[
        (5,5,5,6,6).
\]
From now on we assume $n_3 \geq 4$ or, equivalently, $\tau(\A) \geq 32$.

When $r=4$, one has
$$ \tau(\A) \leq 49-12-1=36,$$
and if $\tau(\A)=36$, then $\A$ is maximal Tjurina of type $(8,4)$, that is nearly free with exponents $(4,4,4)$.

When $\tau(\A)=35$ or equivalently $n_3=7$, then $\A$ is a minimal POG, hence has exponents
$(4,4,5)$, see  \cite[Theorem 1.5]{Brian}.

Consider now the cases $r=4$ and $4 \leq n_3 \leq 6$. Recall that
\[
\dim D_0(f)_5=n_3.
\]
This dimension does not by itself determine the number of minimal generators in degree $5$; one has to use the multiplication map
\[
\mu_4:S_1\otimes AR(f)_4\longrightarrow AR(f)_5.
\]

Suppose first that $n_3=6$. If there are two minimal generators in degree $4$, then the type $1$ classification gives the POG sequence $(4,4,6)$. If instead $\dim AR(f)_4=1$, then $\rank(\mu_4)=3$, and hence there are
\[
6-3=3
\]
new minimal generators in degree $5$. The type $2$ classification then gives the sequence $(4,5,5,5)$. Thus both alternatives in point (3) have $r=4$.

Suppose next that $n_3=5$. The cases with two generators in degree $4$ have Tjurina number at least $34$, so here $\dim AR(f)_4=1$ and $\rank(\mu_4)=3$. Hence there are
\[
5-3=2
\]
new generators in degree $5$, and the degree sequence starts with $(4,5,5)$. The classification of type $2$ curves gives the two possibilities
\[
(4,5,5)\qquad\text{and}\qquad(4,5,5,6),
\]
corresponding respectively to types 2A and 2B. These are cases (4a) and (4b).

Finally, if $n_3=4$ and $r=4$, then $\dim AR(f)_5=4$ and $\dim AR(f)_4=1$, so there is one new generator in degree $5$. The type $2$ classification yields the sequence $(4,5,6,6)$, which is case (5a).

When $r=5$, the type $3$ classification, in particular \cite[Proposition 6.4]{type3}, gives the remaining $r=5$ alternatives listed above: for $n_3=5$ one gets $(5,5,5,5,5)$, while for $n_3=4$ one gets $(5,5,5,5)$ or $(5,5,5,5,6)$. In particular, no $r=5$ case occurs for $n_3=6$. The cases $n_3\leq 3$ were treated above.

\end{proof}

\begin{ex}
The additional branches in Theorem~\ref{thmD} occur.
\begin{enumerate}
\item For
\[
f=xyz(x+y+z)(x+2y+3z)(y+z)(x-y)(x-y-3z),
\]
the six triple points are
\[
127,\quad146,\quad236,\quad378,\quad458,\quad567.
\]
Thus $(n_2,n_3)=(10,6)$ and $\tau(\A)=34$. Moreover,
\[
(\dim AR(f)_3,\dim AR(f)_4,\dim AR(f)_5,\dim AR(f)_6)=(0,1,6,13),
\]
with $\rank(\mu_4)=3$ and $\rank(\mu_5)=13$. Hence
\[
\operatorname{Exp}(f)=(4,5,5,5).
\]

\item For
\[
\begin{aligned}
f={}&(x+2y-3z)(2y-z)(3x-3y+z)(x-2z)\\
&\cdot(2x-3y+3z)(2x+3y)(3x-y)(3x-y-2z),
\end{aligned}
\]
the five triple points are
\[
124,\quad157,\quad237,\quad256,\quad345.
\]
Thus $(n_2,n_3)=(13,5)$ and $\tau(\A)=33$. Moreover,
\[
(\dim AR(f)_3,\dim AR(f)_4,\dim AR(f)_5,\dim AR(f)_6)=(0,1,5,12),
\]
with $\rank(\mu_4)=3$ and $\rank(\mu_5)=11$. Hence
\[
\operatorname{Exp}(f)=(4,5,5,6).
\]
\end{enumerate}
\end{ex}

\begin{cor}
\label{corNoClassical}
For line arrangements with $d<9$, the minimal degree ${\rm mdr}$ is determined by the intersection lattice. In particular, there are no classical Ziegler pairs with fewer than nine lines.
\end{cor}

\begin{proof}
If $m=2$ or $2m>d$, the assertion follows from the formulas in
Section~2. The only remaining cases are
\[
        (d,m)=(6,3),\quad (7,3),\quad (8,4),\quad (8,3).
\]

The case $(6,3)$ is covered by Theorem~\ref{thmA}. For $(8,4)$, the
value of $r={\rm mdr}$ is determined already by the numerical data in
Theorem~\ref{thmC}.

For $(7,3)$, the only weak numerical type for which two values of
${\rm mdr}$ may occur is $n_3=4$. By Theorem~\ref{thmB}(3), the value
is determined by whether the intersection lattice contains a line
with no triple point.

For $(8,3)$, the only weak numerical types for which different values
of ${\rm mdr}$ may occur are those discussed in
Remark~\ref{rkD}. In particular, when $n_3=4$, one has
\[
        r=4
        \quad\Longleftrightarrow\quad
        \text{at least two lines of $\A$ contain no triple point},
\]
while otherwise $r=5$. The corresponding criteria for $n_3=5$ are
also expressed entirely in terms of incidences in the intersection
lattice.

Hence in every case with $d<9$ the value of ${\rm mdr}$ is determined
by the intersection lattice. Therefore there are no classical
Ziegler pairs with fewer than nine lines.
\end{proof}

\section{A two-parameter family of arrangements of \texorpdfstring{$11$}{11} lines}
\label{section5}

We now introduce a two-parameter family of arrangements of eleven lines in \(\mathbb P^2\).  The family is defined by the following \(3\times 11\) realization matrix:
{\small
\[
A(a,b)=
\left(
\begin{array}{ccccccccccc}
1 & 0 & 1 & 1 & b-1 & b-1 & 1 & 1 & 0 & 1 & 0 \\
0 & 1 & 1 & a & ab-1 & ab-1 & a & 1 & 1 & 0 & 0 \\
0 & 0 & 1 & b(1-a) & b(b-1)(1-a) & b-1 & a & 0 & 1-b & b & 1
\end{array}
\right).
\]
}
We regard the columns of \(A(a,b)\) as the linear forms
\[
\ell_1,\ldots,\ell_{11}\in S_1
\]
and set
\[
f_{a,b}=\prod_{i=1}^{11}\ell_i.
\]
Thus \(f_{a,b}=0\) defines an arrangement \(\mathcal A_{a,b}\) of eleven lines.

The following thirteen triple dependencies occur identically in the
family. On the open set defined below, these are exactly the triple
points, the remaining singularities being sixteen double points. The triple points are
{\footnotesize
\[
\begin{aligned}
\mathcal T=\{&
\{\ell_1,\ell_2,\ell_8\},
\{\ell_1,\ell_3,\ell_7\},
\{\ell_1,\ell_{10},\ell_{11}\},
\{\ell_2,\ell_3,\ell_6\},
\{\ell_2,\ell_4,\ell_5\},
\{\ell_2,\ell_9,\ell_{11}\},\\&
\{\ell_3,\ell_8,\ell_{11}\},
\{\ell_3,\ell_9,\ell_{10}\},
\{\ell_4,\ell_7,\ell_{11}\},
\{\ell_4,\ell_8,\ell_{10}\},
\{\ell_5,\ell_6,\ell_{11}\},
\{\ell_5,\ell_8,\ell_9\},
\{\ell_6,\ell_7,\ell_9\}
\}.
\end{aligned}
\]}

The triples listed above correspond to the $3\times 3$ minors of $A(a,b)$ which vanish
identically.  In order to keep the intersection lattice fixed, we must require that no
additional triples occur and that no two columns become proportional.  Equivalently, all
$3\times 3$ minors corresponding to triples not contained in $\mathcal{T}$, together with the
pairwise non-proportionality conditions for the columns, must be nonzero.  After
factorization, the product of the relevant nonzero factors is
\[
\begin{aligned}
\Delta(a,b)=&\,a(a-1)b(b-1)(a-b)(ab-1)\\
&\cdot (ab+a-b)(ab-b+1)\\
&\cdot (ab^2-b^2+b-1)\\
&\cdot (a^2b^2-ab^2-a+b).
\end{aligned}
\]
We work on the open set
\[
        \mathcal{R}=\{(a,b)\in \mathbb A^2\mid \Delta(a,b)\neq 0\}.
\]
For all \((a,b)\in \mathcal{R}\), the arrangements \(\mathcal A_{a,b}\) have \(13\) triple points and \(16\) double points.  Hence
\[
\tau(\mathcal A_{a,b})
=
13(3-1)^2+16(2-1)^2
=
68.
\]

For the rest of this section we write $f=f_{a,b}$. We first consider the concrete point
\[
        q=(a,b)=(31,17).
\]
Since $\Delta(31,17)\neq 0$, this point belongs to $\mathcal R$. We write $\mathcal A_q=\mathcal A_{31,17}$.
A direct linear algebra computation in \texttt{SINGULAR} gives
\[
        AR(f)_d=0\qquad\text{for }d<6,
\]
and
\[
\dim AR(f)_6=1,\qquad
\dim AR(f)_7=5,\qquad
\dim AR(f)_8=13,\qquad
\dim AR(f)_9=23.
\]
Moreover, for the multiplication maps
\[
\mu_k:S_1\otimes AR(f)_k\longrightarrow AR(f)_{k+1}
\]
one obtains
\[
\operatorname{rank}(\mu_6)=3,\qquad
\operatorname{rank}(\mu_7)=12,\qquad
\operatorname{rank}(\mu_8)=23.
\]
It follows that there is one minimal generator in degree $6$, two new
minimal generators in degree $7$, one new minimal generator in degree
$8$, and no new minimal generator in degree $9$. Since $\mathcal A_q$ is essential, the bound $e_j\leq d-2=9$ applies, and this determines the complete degree sequence:
\[
        \operatorname{Exp}(f)=(6,7,7,8).
\]
The corresponding minimal graded free resolution of the Milnor
algebra is
\[
0
\longrightarrow
S(-19)^2
\longrightarrow
S(-16)\oplus S(-17)^2\oplus S(-18)
\longrightarrow
S(-10)^3
\longrightarrow
S
\longrightarrow
M(f)
\longrightarrow
0.
\]

Now consider the special point
\[
        p=\left(\frac{3}{10},\frac{2}{7}\right).
\]
Since
\[
\Delta\left(\frac{3}{10},\frac{2}{7}\right)
=
-\frac{7128}{6565234375}\neq 0,
\]
we have $p\in\mathcal R$, so the intersection lattice does not
degenerate at $p$.

At this point one obtains
\[
\dim AR(f_p)_6=1,\qquad
\dim AR(f_p)_7=5,\qquad
\dim AR(f_p)_8=13,\qquad
\dim AR(f_p)_9=23,
\]
and
\[
\operatorname{rank}(\mu_6)=3,\qquad
\operatorname{rank}(\mu_7)=11,\qquad
\operatorname{rank}(\mu_8)=23.
\]
Therefore there is one minimal generator in degree $6$, two new
minimal generators in degree $7$, two new minimal generators in degree
$8$, and no new generator in degree $9$. Since $\mathcal A_p$ is essential, the bound $e_j\leq d-2=9$ applies, and we obtain the complete degree sequence
\[
        \operatorname{Exp}(f_p)=(6,7,7,8,8).
\]
The corresponding minimal graded free resolution of the Milnor
algebra is
{\footnotesize
\[
0
\longrightarrow
S(-18)\oplus S(-19)^2
\longrightarrow
S(-16)\oplus S(-17)^2\oplus S(-18)^2
\longrightarrow
S(-10)^3
\longrightarrow
S
\longrightarrow
M(f_p)
\longrightarrow
0.
\]}

We justify the shifts in the leftmost free modules of the two resolutions. Since $AR(f)=D_0(f)$ is a reflexive $S$-module of rank $2$, it has projective dimension at most $1$. Hence, if $AR(f)$ has $s$ minimal generators, the leftmost free module in the corresponding resolution of $M(f)$ has rank $s-2$. Write \[ H_{M(f)}(t)=\frac{P(t)}{(1-t)^3}. \] Since the Hilbert function is eventually equal to $\tau(\mathcal A)=68$, one has \[ P(1)=P'(1)=0,\qquad P''(1)=136. \] For $\mathcal A_q$, the middle shifts are $16,17,17,18$. Writing the two leftmost shifts as $b_1,b_2$, the identities above give \[ b_1+b_2=38,\qquad b_1(b_1-1)+b_2(b_2-1)=684, \] hence \[ (b_1,b_2)=(19,19). \] For $\mathcal A_p$, the middle shifts are $16,17,17,18,18$. Writing the three leftmost shifts as $b_1,b_2,b_3$, one obtains \[ b_1+b_2+b_3=56,\qquad \sum_{j=1}^3 b_j(b_j-1)=990, \] whose unique integral solution compatible with the ordering of the shifts is \[ (b_1,b_2,b_3)=(18,19,19). \] Thus the two minimal resolutions displayed above are forced by the computed degree sequences and by $\tau(\mathcal A)=68$.
\noindent
Hence it follows that
\[
(\mathcal A_{31,17},\mathcal A_p)
\]
is a Ziegler pair with the \({\rm MDR}\) property.
\begin{remark}
Although the special arrangement has one additional minimal generator of $AR(f)$ in
degree $8$, the minimal resolution also acquires an additional second syzygy in the
corresponding degree.  Hence the numerator of the Hilbert series remains unchanged:
\[
\begin{aligned}
1-3t^{10}+\big(t^{16}+2t^{17}+2t^{18}\big)
       -\big(t^{18}+2t^{19}\big)
&=1-3t^{10}+t^{16}+2t^{17}+t^{18}-2t^{19}.
\end{aligned}
\]
Therefore both Milnor algebras have the same Hilbert series
\[
        H_{M(f)}(t)=
        \frac{1-3t^{10}+t^{16}+2t^{17}+t^{18}-2t^{19}}{(1-t)^3}.
\]

Its initial values are
{\footnotesize
\[
\begin{array}{c|cccccccccccccccccccc}
k
&0&1&2&3&4&5&6&7&8&9&10&11&12&13&14&15&16&17&18&19\\
\hline
H_{M(f)}(k)
&1&3&6&10&15&21&28&36&45&55&63&69&73&75&75&73&70&68&68&68
\end{array}
\]}
and
\[
H_{M(f)}(k)=68
\qquad
\text{for all }k\geq 17.
\]
Therefore the pair also has the \({\rm HF}\) property.
\end{remark}

We finish the section by briefly indicating how the jumping condition was found.  For a fixed \((a,b)\in \mathcal{R}\), the vector space \(AR(f)_d\) is obtained by solving
\[
\alpha f_x+\beta f_y+\gamma f_z=0,
\qquad
\alpha,\beta,\gamma\in S_d.
\]
Writing
\[
\alpha=\sum_{\lambda}\alpha_\lambda
x^{\lambda_1}y^{\lambda_2}z^{\lambda_3},
\quad
\beta=\sum_{\lambda}\beta_\lambda
x^{\lambda_1}y^{\lambda_2}z^{\lambda_3},
\quad
\gamma=\sum_{\lambda}\gamma_\lambda
x^{\lambda_1}y^{\lambda_2}z^{\lambda_3},
\]
where \(|\lambda|=d\), and comparing coefficients gives a finite-dimensional kernel computation over the ground field.  The map
\[
\mu_7:S_1\otimes AR(f)_7\longrightarrow AR(f)_8
\]
is obtained by multiplying a basis \(h_1,\ldots,h_5\) of \(AR(f)_7\) by \(x,y,z\), and then expressing the resulting syzygies in a basis \(g_1,\ldots,g_{13}\) of \(AR(f)_8\).  The columns of the matrix of \(\mu_7\) are the coordinate vectors of
\[
xh_i,\ yh_i,\ zh_i
\qquad
i=1,\ldots,5.
\]
At the concrete point $q=(31,17)$ this matrix has rank $12$, whereas at
\[
        p=\left(\frac{3}{10},\frac{2}{7}\right)
\]
it has rank $11$. Thus the rank drops by one between these two members, producing the additional minimal generator in degree $8$ at $p$. The \texttt{SINGULAR} script in Appendix A verifies the two displayed ranks. No assertion about the generic rank on the whole parameter space $\mathcal R$ is needed for the construction.
\begin{remark}
For the construction above, it is enough to observe that the rank of
\(\mu_7\) drops at the point
\[
        p=\left(\frac{3}{10},\frac{2}{7}\right).
\]
The question whether this is the only point of the parameter space \(\mathcal{R}\) with this
property is not needed here.
\end{remark}
\section{Deletion technique for constructing new Ziegler pairs of ten lines}
\label{section6}

We now construct a Ziegler pair of ten lines by deleting
\[
\ell_1:x=0
\]
from the arrangement at $q=(31,17)$ and the arrangement at
\[
p=\left(\frac{3}{10},\frac{2}{7}\right).
\]
Let the two eleven-line arrangements be denoted by
\[
\mathcal A_q=\mathcal A_{31,17}
\qquad\text{and}\qquad
\mathcal A_p,
\]
and define
\[
\mathcal B_q
=
\mathcal A_q\setminus\{\ell_1\},
\qquad
\mathcal B_p
=
\mathcal A_p\setminus\{\ell_1\}.
\]
Deleting \(\ell_1\) removes exactly the three triple points
\[
\{\ell_1,\ell_2,\ell_8\},
\qquad
\{\ell_1,\ell_3,\ell_7\},
\qquad
\{\ell_1,\ell_{10},\ell_{11}\}.
\]
The remaining ten triple points are
\[
\begin{aligned}
\mathcal T_{\ell_1}=\{&
\{\ell_2,\ell_3,\ell_6\},
\{\ell_2,\ell_4,\ell_5\},
\{\ell_2,\ell_9,\ell_{11}\},
\{\ell_3,\ell_8,\ell_{11}\},
\{\ell_3,\ell_9,\ell_{10}\},
\{\ell_4,\ell_7,\ell_{11}\},\\&
\{\ell_4,\ell_8,\ell_{10}\},
\{\ell_5,\ell_6,\ell_{11}\},
\{\ell_5,\ell_8,\ell_9\},
\{\ell_6,\ell_7,\ell_9\}
\}.
\end{aligned}
\]
Thus both deleted arrangements have weak combinatorics
\[
(t_2,t_3)=(15,10).
\]
A direct computation in \texttt{SINGULAR} gives, for both deleted arrangements, \[ AR(f)_5=0. \] Since $AR(f)\subset S^3$, this implies $AR(f)_d=0$ for every $d<6$. Moreover, \[ \dim AR(f)_6=3,\qquad \dim AR(f)_7=10,\qquad \dim AR(f)_8=19. \]
For $\mathcal B_q$ one has
\[
        \operatorname{rank}(\mu_6)=9,\qquad
        \operatorname{rank}(\mu_7)=19,
\]
whereas for $\mathcal B_p$ one has
\[
        \operatorname{rank}(\mu_6)=8,\qquad
        \operatorname{rank}(\mu_7)=19.
\]
Hence $\mathcal B_q$ has three minimal generators in
degree $6$, one new minimal generator in degree $7$, and no new
generator in degree $8$. Thus
\[
        \operatorname{Exp}(\mathcal B_q)
        =(6,6,6,7).
\]
For $\mathcal B_p$ there are three minimal generators in degree $6$,
two new minimal generators in degree $7$, and no new generator in
degree $8$, so
\[
        \operatorname{Exp}(\mathcal B_p)
        =(6,6,6,7,7).
\]
Since both deleted arrangements are essential, the bound $e_j\leq d-2=8$ applies, and these computations determine the complete degree sequences. Therefore
\[
        (\mathcal B_q,\mathcal B_p)
\]
is a Ziegler pair of ten lines.

Other deletions lead to further examples.  The weak combinatorics obtained by deleting \(\ell_i\) are summarized in the following table:
\[
\begin{array}{|c|c|c|}
\hline
\text{deleted line} & t_2 & t_3 \\
\hline
\ell_1      & 15 & 10 \\
\ell_2      & 18 & 9  \\
\ell_3      & 18 & 9 \\
\ell_4      & 15 & 10 \\
\ell_5      & 15 & 10 \\
\ell_6      & 15 & 10 \\
\ell_7      & 15 & 10 \\
\ell_8      & 18 & 9  \\
\ell_9      & 18 & 9 \\
\ell_{10}  & 15 & 10 \\
\ell_{11}  & 21 & 8 \\
\hline
\end{array}
\]
For \(i\geq 2\), set
\[
\mathcal B^i_j:=\mathcal A_j\setminus\{\ell_i\},
\qquad
j\in\{q,p\}.
\]
For example,
\[
(\mathcal B^4_{q},\mathcal B^4_p)
\]
is another Ziegler pair of ten lines.  Its intersection lattice differs from that of \((\mathcal B_q,\mathcal B_p)\).  Indeed, for \(\mathcal B_q\) the numbers of triple points lying on the remaining lines are
\[
\begin{array}{c|cccccccccc}
\text{line} & \ell_2 & \ell_3 & \ell_4 & \ell_5 & \ell_6 & \ell_7 & \ell_8 & \ell_9 & \ell_{10} & \ell_{11}\\
\hline
\#\text{ triple points} & 3&3&3&3&3&2&3&4&2&4
\end{array}
\]
with multiset
\[
\{4,4,3,3,3,3,3,3,2,2\}.
\]
For \(\mathcal B^4_{q}\), the corresponding table is
\[
\begin{array}{c|cccccccccc}
\text{line} & \ell_1 & \ell_2 & \ell_3 & \ell_5 & \ell_6 & \ell_7 & \ell_8 & \ell_9 & \ell_{10} & \ell_{11}\\
\hline
\#\text{ triple points} & 3&3&4&2&3&2&3&4&2&4
\end{array}
\]
with multiset
\[
\{4,4,4,3,3,3,3,2,2,2\}.
\]
Since these multisets are invariants of the intersection lattice and
\[
\{4,4,3,3,3,3,3,3,2,2\}
\neq
\{4,4,4,3,3,3,3,2,2,2\},
\]
the two deleted arrangements do not have isomorphic intersection lattices.

\section{Known Ziegler pairs with nine and ten lines}

Let us briefly recall the known examples of Ziegler pairs with nine and ten lines
which are relevant for the present paper. In order to better understand some constructions that we want to present below, let us recall one technical result \cite[Lemma 4.2]{AKP}.

\begin{lemma}
\label{corthm2}
Let $\A_1:f_1=0$ be a line arrangement of $d$ lines with exponents 
$$(e_1, e_2, \ldots, e_s).$$
 Let $L$ be a line avoiding all the multiple points of $\A_1$
and consider the arrangement $\A=\A_1 \cup L$. If $\A_1$ is not a pencil of lines,  then the exponents of $\A:f=0$ are 
$$(e_1+1, e_2+1, \ldots, e_s+1,d-1).$$

\end{lemma}
Observe that  when $\A_1$ is a pencil of lines, then it has exponents $(0,d-1)$ and
the corresponding arrangement $\A$ has exponents $(1,d-1)$. Therefore this case has to be excluded in Corollary \ref{corthm2}.

Now we can start our presentation of Ziegler pairs. 

The classical example is Ziegler's pair of arrangements of nine lines
\cite{Zi}.  We denote the two arrangements by
\[
        \mathcal{A} : g =xyz(x + y + z)
          (2x + y + z)
          (2x + 3y + z)
          (2x + 3y + 4z)
          (3x + 5z)
          (3x + 4y + 5z)=0
\]
and
\[
        \mathcal{A}' : g'=xyz(x + y + z)
           (2x + y + z)
           (2x + 3y + z)
           (2x + 3y + 4z)
           (x + 3z)
           (x + 2y + 3z)=0.
\]
Both arrangements have the same intersection lattice and weak combinatorics
\[
        (n_2,n_3)=(18,6),
\]
but their minimal degrees of Jacobian relations are different.  With the notation
used here,
\[
        {\rm mdr}(g)=6,
        \qquad
        {\rm mdr}(g')=5.
\]
Geometrically, these arrangements arise from the hexagonal construction: one takes
the six sides of a hexagon and adds the three main diagonals.  The value of
\({\rm mdr}\) depends on whether the six vertices of the hexagon lie on a conic.
This construction was revisited by Dimca and Sticlaru \cite{DimcaSticlaru2024},
where the conic condition is related to Pascal's theorem.

There are several natural ten-line extensions of this pair.

First, one can add a line \(L\) which is generic with respect to both arrangements.
Then
\[
        \mathcal A\cup L
        \qquad\text{and}\qquad
        \mathcal A'\cup L
\]
have the same intersection lattice and weak combinatorics
\[
        (n_2,n_3)=(27,6).
\]
The \({\rm mdr}\)-jump persists, shifted by one:
\[
        {\rm mdr}(gL)=7,
        \qquad
        {\rm mdr}(g'L)=6.
\]

Second, one can add a common line through a common double point.  For the explicit
classical pair above, take
\[
        L_{\rm dbl}:5x+4y=0,
\]
and set
\[
        G_{\rm dbl}=g\cdot (5x+4y),
        \qquad
        G'_{\rm dbl}=g'\cdot (5x+4y).
\]
Then the two ten-line arrangements have weak combinatorics
\[
        (n_2,n_3)=(24,7),
\]
and
\[
        {\rm mdr}(G_{\rm dbl})=7,
        \qquad
        {\rm mdr}(G'_{\rm dbl})=6.
\]

Third, one can add a common line through a common triple point. In the same
coordinates, take
\[
        L_{\rm tr}:3x+y+z=0,
\]
and set
\[
        G_{\rm tr}=g\cdot (3x+y+z),
        \qquad
        G'_{\rm tr}=g' \cdot(3x+y+z).
\]
Then the added line turns one old triple point into a quadruple point.  Thus the
weak combinatorics is
\[
        (n_2,n_3,n_4)=(24,5,1).
\]
In this case the two arrangements have the same minimal degree of a Jacobian
relation:
\[
        {\rm mdr}(G_{\rm tr})=6,
        \qquad
        {\rm mdr}(G'_{\rm tr})=6.
\]
Thus this extension is not distinguished by \({\rm mdr}\) itself, and we have to look at the degree sequences which are $(6,7,7,7,7,7)$ and $(6,6,7,7)$, respectively.

There is also a more special ten-line extension described in
Remark~3.1 of \cite{DimcaSticlaru2024}.  In that construction one starts with
Ziegler's arrangement \(A_Z:f_Z=0\), for which the six vertices of the corresponding
hexagon lie on the degenerate conic
\[
        Q_Z:(x-y-z)(2x-y+z)=0.
\]
One adds the line
\[
        L_Z:x-y-z=0,
\]
which is one of the two components of \(Q_Z\).  If
\[
        B_Z=A_Z\cup L_Z:g_Z=0,
\]
then
\[
        {\rm mdr}(g_Z)=5=\frac{10}{2}.
\]
For the corresponding non-conic realization \(A'_Z:f'_Z=0\), one defines
\[
        B'_Z=A'_Z\cup L_Z:g'_Z=0.
\]
Then \(B_Z\) and \(B'_Z\) have the same intersection lattice, but
\[
        {\rm mdr}(g'_Z)=6.
\]
The weak combinatorics of this pair is
\[
        (n_2,n_3,n_4)=(18,3,3).
\]
Indeed, the added line \(L_Z\) passes through three old triple points, turning them
into quadruple points, while the remaining three triple points remain triple.

In \cite{nine}, we constructed a new Ziegler pair of line arrangements of
degree \(d=9\) with common weak combinatorics
\[
(n_2,n_3,n_4)=(9,7,1).
\]
In particular, this pair is not lattice-equivalent to the classical
Ziegler--Yuzvinsky example.  If \(f\) and \(g\) denote the two defining
equations, then
\[
\operatorname{mdr}(f)=4,
\qquad
\operatorname{mdr}(g)=5.
\]

Starting from this pair, one can construct a Ziegler pair of degree \(10\) by
adding suitable lines.  Namely, we set
\[
f_{\mathrm{dou}}=f\cdot (2x-2y-z),
\qquad
g_{\mathrm{dou}}=g\cdot (2y+z).
\]
In both cases, the added line passes through three old double points and
turns them into triple points.  Hence the two resulting arrangements have the
same weak combinatorics (and isomorphic lattices)
\[
(n_2,n_3,n_4)=(9,10,1).
\]
For these two degree \(10\) arrangements, the minimal degree of a Jacobian
relation no longer distinguishes the two realizations.  Indeed, one has
\[
\operatorname{mdr}(f_{\mathrm{dou}})=5,
\qquad
\operatorname{mdr}(g_{\mathrm{dou}})=5.
\]
However, the two arrangements are still homologically distinct: their degree
sequences are
\[
(5,5,7)
\qquad\text{and}\qquad
(5,6,6,6),
\]
respectively. Moreover, one obtains another Ziegler pair by adding a single suitable general line to both nine-line arrangements. For instance, adding
\[
\ell:\ 7x+11y+13z=0
\]
to each arrangement yields a pair of ten-line arrangements with common weak combinatorics
\[
(n_2,n_3,n_4)=(18,7,1).
\]

Another ten-line example comes from the orchard construction,
equivalently from the elliptic matroid \(T_{10}\), which was studied
in \cite{CP,DP2026}. It is represented by a one-parameter family
\(Q_t=0\). We restrict to those parameter values for which the
intersection lattice is the lattice \(T_{10}\). The two distinguished
realizations considered here are
\[
        L_{10}^{1}:Q_3=0,
        \qquad
        L_{10}^{2}:Q_{\sqrt5+3}=0.
\]
They have the same intersection lattice and weak combinatorics
\[
        (n_2,n_3)=(9,12).
\]
Their modules of Jacobian relations have different minimal resolutions:
\[
        0\to S(-8)
        \to S(-6)^{\oplus 2}\oplus S(-5)
        \to D_0(Q_3)\to 0,
\]
whereas
\[
        0\to S(-8)\oplus S(-7)
        \to S(-7)\oplus S(-6)^{\oplus 2}\oplus S(-5)
        \to D_0(Q_{\sqrt5+3})\to 0.
\]
Thus \(L_{10}^{1}\) and \(L_{10}^{2}\) form a Ziegler pair in the homological
sense.  In the terminology of \cite{NH}, \(L_{10}^{1}\) is of type \(2A\), while
\(L_{10}^{2}\) is of type \(2B\).  In this case
\[
        {\rm mdr}(Q_3)={\rm mdr}(Q_{\sqrt5+3})=5,
\]
so the difference is not detected by \({\rm mdr}\), but by the higher part of the
minimal resolution.

Our example in Section \ref{section6} gives another ten-line Ziegler pair of this homological type.  Deleting
\[
        \ell_1:x=0
\]
from the arrangement $\mathcal A_q=\mathcal A_{31,17}$ and from the arrangement
\(\mathcal A_p\), where
\[
        p=\left(\frac{3}{10},\frac{2}{7}\right),
\]
we obtain arrangements
\[
        \mathcal B_q
        =
        \mathcal A_{31,17}\setminus\{\ell_1\},
        \qquad
        \mathcal B_p
        =
        \mathcal A_p\setminus\{\ell_1\}.
\]
They have the same intersection lattice and weak combinatorics
\[
        (n_2,n_3)=(15,10),
\]
but their degree sequences are
\[
        (6,6,6,7)
        \qquad\text{and}\qquad
        (6,6,6,7,7).
\]
Thus
\[
        {\rm mdr}(\mathcal B_q)
        =
        {\rm mdr}(\mathcal B_p)
        =
        6,
\]
and the difference is visible only in the minimal graded free resolution of the
Milnor algebra.

The final pair that we present here was discovered recently by Gabriel Sticlaru. Let $\mathcal{A}_f$ and $\mathcal{A}_g$ be the line arrangements defined by
\[
\begin{aligned}
f={}&(y-z)(y+2z)(2y+z)(x-2y-z)(x-y-2z)\\
&\times (x-y+z)(x+y-z)(x+y+2z)(2x-y-2z)(x-3y-z)
\end{aligned}
\]
and
\[
g=xz(x-y)(x+y)(x-y-z)(y+z)(x-z)(x+y+2z)
(x-2y-z)(x-3y-z).
\]
Then $(\mathcal{A}_f,\mathcal{A}_g)$ is a new Ziegler pair of $10$ lines. Indeed, the two
arrangements have isomorphic intersection lattices. Both have the same weak
combinatorics
\[
(n_2,n_3,n_4)=(15,6,2),
\]
and hence, in particular, the same global Tjurina number
\[
\tau=15+4\cdot 6+9\cdot 2=57.
\]
Moreover, an explicit relabelling of the ten lines identifies all multiple
points together with their incidences, and therefore
\[
L(\mathcal{A}_f)\cong L(\mathcal{A}_g).
\]
On the other hand, their Jacobian syzygy modules have different graded
structures. For $f$ the corresponding minimal free resolution is
\[
0\longrightarrow S(-17)
\longrightarrow S(-14)\oplus S(-15)^2
\longrightarrow S(-9)^3
\longrightarrow S,
\]
so that
\[
\operatorname{mdr}(f)=5
\]
and the exponents are $(5,6,6)$. In contrast, for $g$ one has
\[
0\longrightarrow S(-16)^3
\longrightarrow S(-15)^5
\longrightarrow S(-9)^3
\longrightarrow S,
\]
hence
\[
\operatorname{mdr}(g)=6,
\]
with exponents $(6,6,6,6,6)$. Thus $\mathcal{A}_f$ and $\mathcal{A}_g$
have the same intersection lattice but different graded Jacobian
resolutions. 

\begin{rk}
\label{rkMT}
It is worth emphasizing that the arrangement $\mathcal{A}_{g}$ is a maximal Tjurina curve with $d=10$ and $r=6$ as defined in \eqref{eqMTC}. Consequently, this example shows that the property of being a maximal Tjurina curve is not determined by the combinatorics of the arrangement. Note that in view of \eqref{eqFC} one may restate Terao's Conjecture by saying that for $2r<d$, where $
r=\operatorname{mdr}
$,
the equality $\tau(\A)=\tau(d,r)_{max}$ depends only on the combinatorics of $\A$.
\end{rk}

Summarizing, the relevant examples in degrees \(9\) and \(10\) are:
{\footnotesize
\[
\begin{array}{c|c|c|c}
d
& \text{arrangements}
& \text{weak combinatorics}
& \text{difference}
\\
\hline
9
& \text{classical Ziegler pair}
& (n_2,n_3) = (18,6)
& {\rm mdr}=5 \text{ versus } 6
\\
\hline
9
& \text{Dimca--Pokora}
& (n_2,n_3,n_4) = (9,7,1)
& {\rm mdr}=4 \text{ versus } 5
\\
\hline
10
& \text{classical Ziegler pair + generic line}
& (n_2,n_3) = (27,6)
& {\rm mdr}=6 \text{ versus } 7
\\
\hline
10
& \text{Sticlaru}
& (n_2,n_3,n_4)=(15,6,2)
& {\rm mdr}=5 \text{ versus } 6
\\
\hline
10
& \text{common double--point extension}
& (n_2,n_3) = (24,7)
& {\rm mdr}=6 \text{ versus } 7
\\
\hline
10
& \text{common triple-point extension}
& (n_2,n_3,n_4) = (24,5,1)
& (6,6,7,7) \text{ versus } (6,7,7,7,7,7)
\\
\hline
10
& \text{Remark~3.1 extension}
& (n_2,n_3,n_4) = (18,3,3)
& {\rm mdr}=5 \text{ versus } 6
\\
\hline
10
& T_{10}\text{-matroid }
& (n_2,n_3) = (9,12)
& \text{type }2A \text{ versus type }2B
\\
\hline
10
& \text{Dimca--Pokora + double points}
& (n_2,n_3,n_4) = (9,10,1)
& (5,5,7) \text{ versus } (5,6,6,6)
\\
\hline
10
& \text{Dimca--Pokora + general line}
& (n_2,n_3,n_4) = (18,7,1)
& {\rm mdr}=5 \text{ versus } 6
\\
\hline
10
& \text{present deletion example}
& (n_2,n_3) = (15,10)
& (6,6,6,7) \text{ versus } (6,6,6,7,7)
\end{array}.
\]}

\section*{Funding}
Alexandru Dimca is partially supported from the project ``Singularities and Applications'' - CF 132/31.07.2023 funded by the European Union - NextGenerationEU - through Romania's National Recovery and Resilience Plan.

Piotr Pokora is supported by the National Science Centre (Poland) Sonata Bis Grant 
\textbf{2023/50/E/ST1/00025.} For the purpose of Open Access, the author has applied a CC-BY public copyright license to any Author Accepted Manuscript (AAM) version arising from this submission.
\section*{Acknowledgements}
We would like to thank Gabriel Sticlaru for sharing with us his newly discovered Ziegler pair of $10$ lines and for allowing us to present it here.

Piotr Pokora would like to thank Bartek Naskręcki for his help in optimizing the symbolic computations used in this work. Part of the computational workflow was developed and refined with the assistance of \texttt{Codex}, OpenAI's coding agent.
%%%%%%%%%%%%%%%%%%%%%%%%%%%%%%%%%%%%%%%%%%%%%%%%%%%%%%%%%%%%%%%%%%%%%%%%%%%%%%%%%%%%%%%%%%%%%%%%%%%%%%%%%%%%%%%%%%%%%%%%%%%%%%%%%%%%%%%%%

\bigskip
Alexandru Dimca\\
\noindent
Universit\'e C\^ ote d'Azur, CNRS, LJAD, France and Simion Stoilow Institute of Mathematics,P.O. Box 1-764, RO-014700 Bucharest, Romania.\\
Email: \texttt{Alexandru.DIMCA@univ-cotedazur.fr}

Piotr Pokora\\
\noindent
Department of Mathematics,
University of the National Education Commission Krakow,
Podchor\c a\.zych 2,
PL-30-084 Krak\'ow, Poland. \\
Email: \texttt{piotr.pokora@uken.krakow.pl}
\newpage
\section*{Appendix A}
\begin{verbatim}
////////////////////////////////////////////////////////////////////////
// Rank test for \mu: S_1 \otimes AR(f)_m ---> AR(f)_{m+1}
// Inspect only the four final variables:
//
// dim_m = dim AR(f)_{m};
// dim_m1 = dim AR(f)_{m+1};
// rank_mu ;
// new_gens;
//
// To certify the complete degree sequences used in Sections 5 and 6,
// rerun the script with the following choices of parameters and m.
//
// FULL ARRANGEMENT, reference point q=(a,b)=(31,17):
// m=6 : (dim_m,dim_m1,rank_mu,new_gens)=(1,5,3,2)
// m=7 : (5,13,12,1)
// m=8 : (13,23,23,0)
//
// FULL ARRANGEMENT, special point (a,b)=(3/10,2/7):
// m=6 : (1,5,3,2)
// m=7 : (5,13,11,2)
// m=8 : (13,23,23,0)
//
// DELETION OF L1, reference point q=(a,b)=(31,17):
// m=5 : (dim_m,dim_m1,rank_mu,new_gens)=(0,3,0,3)
// m=6 : (3,10,9,1)
// m=7 : (10,19,19,0)
//
// DELETION OF L1, special point (a,b)=(3/10,2/7):
// m=5 : (dim_m,dim_m1,rank_mu,new_gens)=(0,3,0,3)
// m=6 : (3,10,8,2)
// m=7 : (10,19,19,0)
////////////////////////////////////////////////////////////////////////

LIB "matrix.lib";

ring R = 0,(x,y,z),dp;

// !Choose parameters!

// Reference point q:
// number a = 31;
// number b = 17;

// Special point:
number a = 3/10;
number b = 2/7;

// Lines.

poly L1  = x;
poly L2  = y;
poly L3  = x + y + z;
poly L4  = x + a*y + b*(1-a)*z;
poly L5  = (b-1)*x + (a*b-1)*y + b*(b-1)*(1-a)*z;
poly L6  = (b-1)*x + (a*b-1)*y + (b-1)*z;
poly L7  = x + a*y + a*z;
poly L8  = x + y;
poly L9  = y + (1-b)*z;
poly L10 = x + b*z;
poly L11 = z;

// !Choose arrangement!

// Full arrangement:
 poly f = L1*L2*L3*L4*L5*L6*L7*L8*L9*L10*L11;
 int m = 7;

// For deletion of L1, use instead:
// poly f = L2*L3*L4*L5*L6*L7*L8*L9*L10*L11;
// int m = 5;  // rerun also with m=6 and m=7

poly fx = diff(f,x);
poly fy = diff(f,y);
poly fz = diff(f,z);

// Coefficient of monomial q in polynomial p.

proc cc(poly p, poly q)
{
    number c = 0;
    while (p != 0)
    {
        if (leadmonom(p) == q)
        {
            c = leadcoef(p);
            return(c);
        }
        p = p - lead(p);
    }
    return(c);
}

// Matrix whose kernel is AR(f)_d.

proc ARmat(int d)
{
    ideal Md = maxideal(d);
    ideal Mt = maxideal(d + deg(f) - 1);

    int nd = size(Md);
    int nt = size(Mt);

    matrix A[nt][3*nd];

    int i,j;

    for (i=1; i<=nd; i++)
    {
        for (j=1; j<=nt; j++)
        {
            A[j,i]        = cc(Md[i]*fx, Mt[j]);
            A[j,nd+i]     = cc(Md[i]*fy, Mt[j]);
            A[j,2*nd+i]   = cc(Md[i]*fz, Mt[j]);
        }
    }

    return(A);
}

// Compute AR(f)_m and AR(f)_{m+1}.

matrix A0 = ARmat(m);
matrix A1 = ARmat(m+1);

module K0 = syz(A0);
module K1 = syz(A1);

int dim_m  = ncols(K0);
int dim_m1 = ncols(K1);

// Build the image of S_1 * AR(f)_m in S_{m+1}^3.

ideal Mm  = maxideal(m);
ideal Mm1 = maxideal(m+1);

int nm  = size(Mm);
int nm1 = size(Mm1);

matrix Images[3*nm1][3*dim_m];

int i,j,k;

for (i=1; i<=dim_m; i++)
{
    poly P = 0;
    poly Q = 0;
    poly H = 0;

    for (j=1; j<=nm; j++)
    {
        P = P + K0[j,i]*Mm[j];
        Q = Q + K0[nm+j,i]*Mm[j];
        H = H + K0[2*nm+j,i]*Mm[j];
    }

    for (k=1; k<=nm1; k++)
    {
        Images[k,3*i-2]       = cc(x*P, Mm1[k]);
        Images[nm1+k,3*i-2]   = cc(x*Q, Mm1[k]);
        Images[2*nm1+k,3*i-2] = cc(x*H, Mm1[k]);

        Images[k,3*i-1]       = cc(y*P, Mm1[k]);
        Images[nm1+k,3*i-1]   = cc(y*Q, Mm1[k]);
        Images[2*nm1+k,3*i-1] = cc(y*H, Mm1[k]);

        Images[k,3*i]         = cc(z*P, Mm1[k]);
        Images[nm1+k,3*i]     = cc(z*Q, Mm1[k]);
        Images[2*nm1+k,3*i]   = cc(z*H, Mm1[k]);
    }
}

int rank_mu = rank(Images);
int new_gens = dim_m1 - rank_mu;

// Check these four values.

dim_m;
dim_m1;
rank_mu;
new_gens;
\end{verbatim}

\end{document}